\documentclass[11pt,a4paper]{article}
\usepackage{fullpage}

\usepackage[utf8]{inputenc}
\usepackage{amsthm,amssymb}
\usepackage[leqno]{amsmath}
\usepackage{tcolorbox}
\usepackage{thmtools}
\usepackage{mathtools}
\usepackage{subcaption,thm-restate}
\usepackage[mathcal,mathscr]{eucal}
\usepackage{epsfig,graphicx,graphics,color}
\usepackage{caption}
\usepackage{enumerate}
\usepackage{enumitem, crossreftools}
\usepackage[sort]{cite}
\usepackage{hyperref}
\usepackage{tikz}
\hypersetup{
    colorlinks=true,       % false: boxed links; true: colored links
    linkcolor=blue,        % color of internal links (change box color with linkbordercolor)
    citecolor=red,         % color of links to bibliography
    filecolor=magenta,     % color of file links
    urlcolor=cyan,         % color of external links
    linktocpage=true
}

\usepackage{enumitem}
\newenvironment{symenum}
 {\enumerate[label=(\noexpand\thisenumsymbol),align=parleft,labelindent=0pt,itemindent=0pt,labelsep=35pt,leftmargin=*]}
 {\endenumerate}
\newcommand\thisenumsymbol{}
\newcommand\itemsymbol[1]{%
  \renewcommand{\thisenumsymbol}{#1}%
  \item
}

\theoremstyle{plain}
\newtheorem{thm}{Theorem}
\newtheorem{lem}[thm]{Lemma}
\newtheorem{cor}[thm]{Corollary}

\newtheorem{conj}{Conjecture}
\newtheorem{qu}{Question}

\theoremstyle{definition}

\newtheorem{rem}[thm]{Remark}

\def\final{0}  % set this to 1 to get a comment-free version
\def\iflong{\iffalse}
\ifnum\final=0  %namely if we allow comments in the output
\newcommand{\knote}[1]{{\color{red}[{\tiny \textbf{Kristóf:} \bf #1}]\marginpar{\color{red}*}}}
\newcommand{\tnote}[1]{{\color{blue}[{\tiny \textbf{Tamás:} \bf #1}]\marginpar{\color{blue}*}}}
\else % in this case [final=1] we don't want any comments to show
\newcommand{\knote}[1]{}
\newcommand{\tnote}[1]{}
\fi

\DeclareMathOperator*{\exx}{Ex}
\DeclareMathOperator*{\sbo}{SBO}
\DeclareMathOperator*{\bo}{BO}
\DeclareMathOperator*{\sbro}{SBRO}

\newcommand{\cB}{\mathcal{B}}
\newcommand{\cI}{\mathcal{I}}

\newcommand{\cH}{\mathcal{H}}

\newcommand{\cC}{\mathcal{C}}

\newcommand{\cM}{\mathcal{M}}
\newcommand{\exk}{\exx(M(K_4))}

\newcommand{\phiext}{\hat{\varphi}}

\title{Partitioning into common independent sets\\ via relaxing strongly base orderability}

\author{
Kristóf Bérczi\thanks{MTA-ELTE Momentum Matroid Optimization Research Group and ELKH-ELTE Egerváry Research Group, Department of Operations Research, Eötvös Loránd University, Budapest, Hungary. Email: \texttt{kristof.berczi@ttk.elte.hu, tamas.schwarcz@ttk.elte.hu}.}
\and
Tamás Schwarcz\footnotemark[1]
}

\date{}

\begin{document}
\maketitle
%%%%%%%%%%%%%%%%%%%%%%%%

\begin{abstract}

The problem of covering the ground set of two matroids by a minimum number of common independent sets is notoriously hard even in very restricted settings, i.e.\ when the goal is to decide if two common independent sets suffice or not. Nevertheless, as the problem generalizes several long-standing open questions, identifying tractable cases is of particular interest. Strongly base orderable matroids form a class for which a basis-exchange condition that is much stronger than the standard axiom is met. As a result, several problems that are open for arbitrary matroids can be solved for this class. In particular, Davies and McDiarmid showed that if both matroids are strongly base orderable, then the covering number of their intersection coincides with the maximum of their covering numbers.

Motivated by their result, we propose relaxations of strongly base orderability in two directions. First we weaken the basis-exchange condition, which leads to the definition of a new, complete class of matroids with distinguished algorithmic properties. Second, we introduce the notion of covering the circuits of a matroid by a graph, and consider the cases when the graph is ought to be 2-regular or a path. We give an extensive list of results explaining how the proposed relaxations compare to existing conjectures and theorems on coverings by common independent sets. 
   
    \medskip

    
\end{abstract}

%%%%%%%%%%%%%%%%
\section{Introduction} 
\label{sec:intro}
%%%%%%%%%%%%%%%%

For basic definitions and notation of matroid theory, the interested reader is referred to \cite{oxley2011matroid}. Throughout the paper, we denote the ground set of a matroid $M$ by $E$ while the sets of independent sets, bases and circuits are denoted by $\cI$, $\cB$ and $\cC$, respectively. The \textbf{covering number} $\beta(M)$ of a matroid $M$ is the minimum number of independent sets needed to cover its ground set. A matroid is then called \textbf{$k$-coverable} if $\beta(M)\leq k$. Whenever investigating the covering number, we assume the matroid to be loopless as otherwise the ground set obviously cannot be covered by independent sets. The value of $\beta(M)$ can be determined using the rank formula of the union of matroids due to Edmonds and Fulkerson~\cite{edmonds1965transversals}. 

It is quite natural to consider an analogous notion for the intersection of two matroids. Given two matroids $M_1$ and $M_2$ on the same ground set, the \textbf{covering number $\beta(M_1\cap M_2)$ of their intersection} is the minimum number of common independent sets needed to cover the common ground set. Determining the exact value of $\beta(M_1\cap M_2)$ has been the center of attention for a long time since it generalizes a wide list of fundamental questions from both graph and matroid theory, including Woodall's conjecture~\cite{woodall1978menger} on the maximum number of pairwise disjoint dijoins in directed graphs, or Rota's basis conjecture~\cite{huang1994relations} on packing transversal bases. Nevertheless, apart from partial results such as K\H{o}nig's $1$-factorization theorem~\cite{konig1916graphen} or Edmonds' disjoint arborescences theorem~\cite{edmonds1973edge}, the problem remained open until recently, when the authors settled the complexity of the problem by showing hardness under the rank oracle model~\cite{berczi2021complexity}.

As determining the exact value of $\beta(M_1\cap M_2)$ is hard in general, the need for good lower and upper bounds arises. A lower bound is easy to give as $\beta(M_1\cap M_2)\geq\min\{\beta(M_1),\beta(M_2)\}$ always holds. Nevertheless, the equality $\beta(M_1\cap M_2) = \max\{\beta(M_1), \beta(M_2)\}$ does not necessarily hold for general matroids, as shown by the well-known example where $M_1$ is the graphic matroid of a complete graph on four vertices and $M_2$ is the partition matroid\footnote{All partition matroids considered in the paper have all-ones upper bounds on the partition classes without explicitly mentioning it.} defined by the partition of its edges into three matchings, see~\cite{schrijver2003combinatorial} for details. As for the upper bound, Aharoni and Berger~\cite{aharoni2006intersection} showed by using techniques from topology that $\beta(M_1\cap M_2)\leq 2\max\{\beta(M_1),\beta(M_2)\}$. Furthermore, they verified the slightly stronger statement that $\beta(M_1\cap M_2) \le \beta(M_1)+\beta(M_2)$ holds whenever one of $\beta(M_1)$ and $\beta(M_2)$ divides the other. Nevertheless, no example is known for which the true value would be close to the upper bound. In fact, Aharoni and Berger~\cite{aharoni2012edge} conjectured the following, originally attributed to \cite{aharoni2006intersection}.

\begin{conj}{\normalfont(Aharoni and Berger)}\label{conj:ab}
Let $M_1$ and $M_2$ be matroids on the same ground set.
\begin{enumerate}[label={(\arabic*)}]\itemsep0em
    \item If $\beta(M_1)\neq \beta(M_2)$, then $\beta(M_1\cap M_2)=\max\{\beta(M_1),\beta(M_2)\}$.
    \item If $\beta(M_1)= \beta(M_2)$, then $\beta(M_1\cap M_2)\leq \max\{\beta(M_1),\beta(M_2)\}+1$.
\end{enumerate} 
\end{conj}
The conjecture was verified only for $\beta(M_1)=\beta(M_2)=2$ by Aharoni, Berger and Ziv~\cite{aharoni2012edge}. In the same paper, the authors also showed that if $\beta(M_1)=2$ and $\beta(M_2) = 3$, then $\beta(M_1 \cap M_2) \le 4$ holds. For $\beta(M_1)=2$ and $\beta(M_2) = k \ge 4$, the current best bound follows from the result of Aharoni and Berger~\cite{aharoni2006intersection} mentioned above: $\beta(M_1 \cap M_2) \le k+2$ if $k$ is even, and $\beta(M_1 \cap M_2) \le k+3$ if $k$ is odd.   

Among the results related to Conjecture~\ref{conj:ab}, the probably most important one is due to Davies and McDiarmid~\cite{davies1976disjoint}, who studied the class of strongly base orderable matroids. A matroid is called \textbf{strongly base orderable} if
\begin{symenum}
\itemsep0em
    \itemsymbol{SBO}\label{prop:sbo} for any pair $A,B$ of bases, there exists a bijection $\varphi\colon A\setminus B \to B\setminus A$ such that $(A\setminus X)\cup \varphi(X)$ is a basis for every $X\subseteq A\setminus B$.
\end{symenum}
It is worth mentioning that this implies $(B\cup X)\setminus \varphi(X)$ being a basis as well. Strongly base orderable matroids are interesting and important because we have a fairly good global understanding of their structure, while frustratingly little is known about the general case. In particular, Davies and McDiarmid showed that the covering number of the intersection of two strongly base orderable matroids coincides with the obvious lower bound.

\begin{thm}{\normalfont(Davies and McDiarmid)}\label{thm:dm} 
Let $M_1$ and $M_2$ be strongly base orderable matroids on the same ground set. Then $\beta(M_1\cap M_2) = \max\{\beta(M_1), \beta(M_2)\}$.
\end{thm}

As many matroid classes that naturally appear in combinatorial and graph optimization problems, e.g.\ gammoids, are strongly base orderable, Theorem~\ref{thm:dm} was a milestone result in the research on packing common independent sets. Recently, the theorem received a renewed interest when Abdi, Cornu\'ejols and Zlatin~\cite{abdi2022packing} successfully attacked special cases of Woodall's conjecture with its help. However, there are basic matroid classes that do not satisfy strongly base orderability, e.g.\ graphic or paving matroids.   

\paragraph{Our contribution.} Our research was motivated by the following question: can strongly base orderability in Theorem~\ref{thm:dm} replaced with some weaker assumption so that the statement remains true? Or more generally, can we verify Conjecture~\ref{conj:ab} for some reasonably broad class of matroids? As partial answers to these problems, we propose two relaxations of strongly base orderability that weaken the original definition in different aspects. 

First, in Section~\ref{sec:sbro} we omit the condition for the bijection to go between $A\setminus B$ and $B\setminus A$, and allow repartitioning the multiset $A\cup B$ into two new bases $A'$ and $B'$ satisfying $A'\cap B'=A\cap B$ and $A'\cup B'=A\cup B$. We show that matroids satisfying the relaxed condition form a complete class. We also prove that Theorem~\ref{thm:dm} remains true when both matroids are from the proposed class, and thus we identify new scenarios when packing problems are tractable. Finally, we explain how the new class fits in the hierarchy of existing matroid classes.

Second, in Section~\ref{sec:path} we relax the condition of finding a bijection between $A\setminus B$ and $B\setminus A$. Instead, we seek for a graph whose vertex set coincides with the symmetric difference of the bases, and the graph represents the matroid locally in the sense that every stable set in it corresponds to an independent set of the original matroid. In particular, we prove that the the graph in question can be chosen to be a path for graphic matroids, paving matroids, and spikes. We conjecture that an analogous statement hold for arbitrary matroids, and discuss a series of corollaries that would follows from such a result. 

%%%%%%%%%%%%%%%%
\section{Relaxing the fixed bases: strongly base reorderable matroids} 
\label{sec:sbro}
%%%%%%%%%%%%%%%%

When considering the extendability of Theorem~\ref{thm:dm} to a broader class of matroids, some natural candidates are immediate. For subsets $X,Y\subseteq E$, their \textbf{symmetric difference} is defined as $X\triangle Y\coloneqq (X\setminus Y)\cup (Y\setminus X)$. When $Y$ consist of a single element $y$, then $X\setminus \{y\}$ and $X\cup\{y\}$ are abbreviated as $X-y$ and $X+y$, respectively. A matroid is called \textbf{base orderable} if 
\begin{symenum}
\itemsep0em
    \itemsymbol{BO}\label{prop:bo} for any pair $A,B$ of bases, there exists a bijection $\varphi\colon A\setminus B \to B\setminus A$ such that $A-x+\varphi(x)$ and $B+x-\varphi(x)$ are bases for every $x\in A\setminus B$.
\end{symenum}
Clearly, strongly base orderable matroids are also base orderable since the bijection appearing in \ref{prop:sbo} satisfies \ref{prop:bo} as well. It is known that $M(K_4)$, the graphic matroid of a complete graph on four vertices is a forbidden minor for base orderability. Hence every base orderable matroid is contained in the class $\exk$ of $M(K_4)$-minor-free matroids. In fact, these inclusions are known to be strict, hence $\sbo\subsetneq\bo\subsetneq\exk$ hold. 

Davies and McDiarmid~\cite{davies1976disjoint} posed the problem of whether Theorem~\ref{thm:dm} remains true if strongly base orderability is replaced with the weaker assumption that both matroids are base orderable. Though this specific question remains open, the following example shows that both matroids being in $\exk$ does not suffice. The examples uses $J$, a self-dual rank-$4$ matroid introduced by Oxley~\cite{oxley1987characterization}.

\begin{rem}\label{rem:antidm}
\normalfont
We show that there exists a partition matroid $M$ for which $\beta(J)=\beta(M)=2$ and $\beta(J\cap M)=3$. Consider the geometric representation of the matroid $J$ on Figure~\ref{fig:J}; for the definition, see also \cite[page 650]{oxley2011matroid}. Let $M$ be the partition matroid defined by partition classes $\{a,h\}, \{b,g\}, \{c,e\}, \{d,f\}$, implying $\beta(J)=\beta(M)=2$.

To see that $\beta(J\cap M) > 2$, suppose to the contrary that $E=B_1 \cup B_2$ is a decomposition into two common bases of $J$ and $M$. Without loss of generality, we may assume that $h \in B_1$. As $B_1$ is a common basis, it contains exactly one element from each of the pairs $\{b,g\}$, $\{c,e\}$, $\{d,f\}$ and at most one element from each of the pairs $\{b,c\}$, $\{d,g\}$, $\{e,f\}$, hence $B_1 = \{b,d,e,h\}$ or $B_1 = \{c,f,g,h\}$. In either case, $B_2$ is not a basis of $J$, a contradiction. Therefore $\beta(J\cap M)>2$, and so $\beta(J\cap M)= 3$ by the result of Aharoni, Berger and Ziv mentioned earlier.  
\end{rem}

\begin{figure} \centering 
\includegraphics[width=0.5\textwidth]{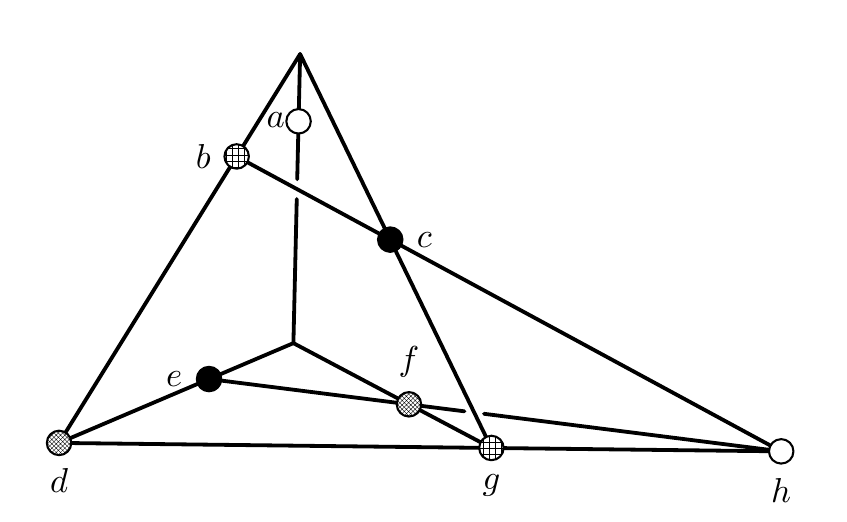}
\caption{Geometric representation of the matroid $J$: bases are the sets of size four which do not lie on a plane. If $M_1= J$ and $M_2$ is the partition matroid defined by partition classes $\{a,h\}, \{b,g\}, \{c,e\}, \{d,f\}$ with upper bounds one, then $\beta(M_1)=\beta(M_2)=2$ and $\beta(M_1\cap M_2)=3$.}
\label{fig:J}
\end{figure}

Motivated by the proof of Theorem~\ref{thm:dm}, we call a matroid \textbf{strongly base reorderable} (or SBRO for short) if 
\begin{symenum}
\itemsep0em
    \itemsymbol{SBRO}\label{prop:sbro} for any pair $A,B$  of bases, there exists bases $A'$, $B'$ and a bijection $\varphi\colon A'\setminus B' \to B'\setminus A'$ such that $A' \cap B' = A\cap B$, $A' \cup B' = A \cup B$, and $(A'\setminus X)\cup\varphi(X)$ is a basis for every $X \subseteq A'\setminus B'$. 
\end{symenum}
In other words, \ref{prop:sbro} differs from \ref{prop:sbo} in that it allows the repartitioning of the multiunion of the bases before asking for the bijection $\varphi$. It immediately follows from the definition that strongly base orderable matroids are strongly base reorderable as well.

%%%%%%%%%%%%%%%%
\subsection{Completeness and intersection}
%%%%%%%%%%%%%%%%

Ingleton~\cite{ingleton1977transversal} defined a class of matroids to be \textbf{complete} if it is closed under taking minors, duals, direct sums, truncations and induction by directed graphs. A complete class is also closed under many other matroid operations, such as series and parallel connections, 2-sums, unions and principal extensions, see e.g.~\cite{bonin2016infinite}. It was already noted by Ingleton~\cite{ingleton1977transversal} that the classes of base orderable and strongly base orderable matroids are complete, while Sims~\cite{sims1977complete} verified that $\exk$ is complete as well. 

A \textbf{$k$-exchange-ordering} for a pair $A, B$ of bases is a bijection $\varphi\colon A\setminus B \to B \setminus A$ such that both $(A\setminus X)\cup \varphi(X)$ and $(B \setminus \varphi(X)) \cup X$ are bases for each subset $X \subseteq A\setminus B$ with $|X| \le k$. A matroid is called \textbf{$k$-base-orderable} if each pair of its bases has a $k$-exchange-ordering. Note that $k$-base-orderability corresponds to base orderability for $k=1$ and to strongly base orderability if $k$ is the rank of the matroid. Bonin and Savitsky~\cite{bonin2016infinite} showed that the class of $k$-base-orderable matroids is complete for any fixed $k \ge 1$. Relying on their proof, we verify that the class $\sbro$ of strongly base reorderable matroids is complete as well. Following the notation of \cite{oxley2011matroid}, the \textbf{contraction} and \textbf{deletion} of a set $X$ in a matroid $M$ are denoted by $M/X$ and $M\backslash X$, respectively.

\begin{thm}\label{thm:complete}
$\sbro$ is a complete class.
\end{thm}
\begin{proof}
Bonin and Savitsky~\cite{bonin2016infinite} verified that a class of matroids is complete if it is closed under taking minors, duals, direct sums and principal extensions into flats. For a flat $F$ of a matroid $M$ over ground set $E$ and an element $e$ not in $E$,  the \textbf{principal extension} $M+_{F} e$ is a matroid on ground set $E+e$ with family of bases \[\{B+e-f \mid B\text{ is a basis of $M$, } f \in (B \cap F) + e\}.\]
The following statement was implicitly showed in \cite{bonin2016infinite} in the proof that the class of $k$-base-orderable matroids is closed under principal extensions.
\begin{lem} \label{lem:principal}
Let $k$ be a positive integer, $F$ be a flat of a matroid $M$ over ground set $E$, and $e$ be an element not in $E$. If $\varphi_0$ is a $k$-exchange-ordering for bases $A_0$ and $B_0$ of $M$, then for each $x \in (A_0 \cap F)+e$ and $y \in (B_0 \cap F)+e$, there exists a $k$-exchange-ordering for bases $A_0 + e - x$ and $B_0 + e - y$ of $M+_{F} e$.   
\end{lem}

We prove Theorem~\ref{thm:complete} by showing that $\sbro$ is closed under taking direct sums, duals, minors and principal extensions. Closedness under taking direct sums is straightforward to verify. 

To prove closedness under taking duals, let $A^*$ and $B^*$ be bases of the dual $M^*$ of an SBRO matroid $M$. Since $A \coloneqq E\setminus A^*$ and $B \coloneqq E\setminus B^*$ are bases of $M$, there exist bases $A', B'$ and a bijection $\varphi\colon A' \setminus B' \to B' \setminus A'$ such that $A' \cap B' = A \cap B$, $A' \cup B' = A \cup B$, and $(A' \setminus X) \cup \varphi(X)$ is a basis of $M$ for every $X \subseteq A' \setminus B'$. Then $A'^{*} \coloneqq E \setminus A'$ and $B'^{*} \coloneqq E \setminus B'$ are bases of $M^*$ such that $A'^{*} \cap B'^{*} = A^* \cap B^*$, $A'^* \cup B'^* = A\cup B$, and $\varphi^{-1}$ is a bijection from $B' \setminus A' = A'^{*} \setminus B'^{*}$ to $A' \setminus B' = B'^{*} \setminus A'^{*}$ such that for every $Y \subseteq B' \setminus A'$ the set $(A'^{*} \setminus Y) \cup \varphi^{-1}(Y)$ is a basis of $M$ since $E \setminus \left((A'^{*} \setminus Y) \cup \varphi^{-1}(Y)\right) = (A'\cup Y) \setminus \varphi^{-1}(Y) = (A'\setminus \varphi^{-1}(Y)) \cup \varphi\left(\varphi^{-1}(Y)\right)$ is a basis of $M$.

Next we show closedness under taking minors. Observe that the matroid $M \backslash x$ obtained from an SBRO matroid $M$ by deleting an element $x$ is SBRO. This follows from the fact that if $x$ is not a coloop of $M$, then each basis of $M \backslash x$ is a basis of $M$, while if $x$ is a coloop of $M$, then $B+x$ is a basis of $M$ for every basis $B$ of $M\backslash x$. As we have already seen that $\sbro$ is closed under taking duals, it follows that the contraction $M/x = (M^*\backslash x)^*$ is SBRO as well.     

Finally, we show closedness under principal extensions. Let $F$ be a flat of an SBRO matroid $M$, and consider a pair $A, B$ of bases of $M+_{F} e$. By the definition of the principal extension there exist $x \in F \cap A+e$ and $y \in F\cap B + e$ such that $A+x-e$ and $B+y-e$ are bases of $M$. Since $M$ is SBRO, there exist a pair $A_0, B_0$ and a bijection $\varphi_0\colon A_0 \setminus B_0 \to B_0 \setminus A_0$ such that $A_0 \cap B_0 = (A+x-e) \cap (B+y-e)$, $A_0 \cup B_0 = (A+x-e) \cup (B+y-e)$ and $(A_0 \setminus X) \cup \varphi(X)$ is a basis of $M$ for every $X \subseteq A_0 \setminus B_0$. Since $\varphi_0$ is a $k$-exchange ordering for $k$ being the rank of $M$, there exists a $k$-exchange ordering $\varphi$ for the bases of $A' \coloneqq A_0+e-x$ and $B' \coloneqq B_0+e-y$ of $M+_F e$ by Lemma~\ref{lem:principal}. As $A' \cap B' = A \cap B$ and $A' \cup B' = A \cup B$ hold, the pair $A', B'$ and the bijection $\varphi$ satisfies the requirements of \ref{prop:sbro}. 
\end{proof}

With the help of Theorem~\ref{thm:complete}, we are ready to prove the next theorem which shows the importance of the proposed class. Though the proof is analogous to that of Theorem~\ref{thm:dm} appearing in~\cite[Theorem~42.13]{schrijver2003combinatorial}, we include it to make the paper self-contained. 

\begin{thm}\label{thm:sbro} Let $M_1$ and $M_2$ be strongly base reorderable matroids on the same ground set. Then $\beta(M_1\cap M_2) = \max\{\beta(M_1),\beta(M_2)\}$.
\end{thm}
\begin{proof}
Let $k = \max \{\beta(M_1), \beta(M_2)\}$ and $r = \max\{r(M_1), r(M_2)\}$. We may assume that the common ground set $E$ of the matroids decomposes into $k$ disjoint bases in each of the matroids $M_1$ and $M_2$. Indeed, for $i=1,2$, consider the matroid $M'_i$ obtained as the $r$-truncation of the direct sum of $M_i$ and the free matroid on $k\cdot r-|E|$ elements. Then $M'_i$ is SBRO by Theorem~\ref{thm:complete} and its ground set decomposes into $k$ disjoint bases for $i=1,2$,  and $\beta(M'_1 \cap M'_2) = \beta(M_1 \cap M_2)$ holds as well.

Let $\mathcal{X} = \{X_1, \dots, X_k\}$ and $\mathcal{Y} = \{Y_1, \dots, Y_k\}$ be partitions of $E$ into bases of $M_1$ and $M_2$, respectively, such that $\sum_{i=1}^k |X_i \cap Y_i|$ is as large as possible. If it has value $|E|$, then $\mathcal{X} = \mathcal{Y}$ is a decomposition of $E$ into $k$ common bases of $M_1$ and $M_2$. Otherwise, there exist distinct indices $i$ and $j$ such that $X_i \cap Y_j \ne \emptyset$. 
Let $X'_i$ and $X'_j$ be bases of $M_1$ and $\varphi_1\colon X'_i \to X'_j$ be a bijection guaranteed by \ref{prop:sbro} for the disjoint bases $X_i$ and $X_j$ of $M_1$. Similarly, let $Y'_i$ and $Y'_j$ be bases of $M_2$ and $\varphi_2\colon Y'_i \to Y'_j$ be a bijection guaranteed by \ref{prop:sbro} for the disjoint bases $Y_i$ and $Y_j$ of $M_2$. Consider the graph $G$ on $X_i \cup X_j \cup Y_i \cup Y_j$ with edge set $\{e\varphi_1(e) \mid e \in X'_i\} \cup \{e\varphi_2(e) \mid e \in Y'_i\}$. Since $G$ is the union of two matchings, it is bipartite. Let $S$ and $T$ denote the color classes of a 2-coloring of the vertices of $G$. As $(X'_i \setminus Z) \cup \varphi_1(Z)$ is a basis of $M_1$ for every $Z \subseteq X'_i$ and $(Y'_i \setminus Z) \cup \varphi_2(Z)$ is a basis of $M_2$ for every $Z \subseteq Y'_i$, $X''_i \coloneqq S\cap (X_i \cup X_j)$ and $X''_j \coloneqq T \cap (X_i \cup X_j)$ are bases of $M_1$, while $Y''_i \coloneqq S \cap (Y_i \cup Y_j)$ and $Y''_j \coloneqq T \cap (Y_i \cup Y_j)$ are bases of $M_2$. Consider the partition $\mathcal{X}'$ obtained from $\mathcal{X}$ by replacing $X_i$ and $X_j$ by $X''_i$ and $X''_j$, respectively, and the partition $\mathcal{Y}'$ obtained from $\mathcal{Y}$ by replacing $Y_i$ and $Y_j$ by $Y''_i$ and $Y''_j$, respectively. Then $|X''_i \cap Y''_i| + |X''_j \cap Y''_j| = |(X_i \cup X_j) \cap (Y_i \cup Y_j)| > |X_i \cap Y_i|  + |X_j \cap Y_j|$, hence $\mathcal{X}'$ and $\mathcal{Y}$' contradict the maximality of $\sum_{i=1}^k |X_i \cap Y_i|$.
\end{proof}

%%%%%%%%%%%%%%%%
\subsection{Binary and ternary SBRO matroids}
%%%%%%%%%%%%%%%%

In general, knowing the excluded minors for a minor-closed matroid class provides a powerful tool that then can be used in various applications.  Based on the characterization of $M(K_4)$-minor-free binary matroids by Brylawski~\cite{brylawski1971combinatorial} and of $M(K_4)$-minor-free ternary matroids by Oxley~\cite{oxley1987characterization}, we give a characterization of binary and of ternary SBRO matroids. For the latter we will use that the matroid $S(5,6,12)$ is SBRO. The \textbf{Steiner system $S(5,6,12)$} is a family of 6-element subsets of a 12-element ground set such that each 5-element subset of the ground set is contained in exactly one member of the family. By slight abuse of notation, the matroid whose family of hyperplanes is identical with such a system is denoted by $S(5,6,12)$ as well.

\begin{lem}\label{lem:issbro}
    $S(5,6,12)$ is SBRO.
\end{lem}
\begin{proof} 
As a preparation for the proof, we discuss representations of some minors of $S(5,6,12)$ over the field $GF(3)$. We denote the columns of an $r \times n$ matrix $A$ by $a_1, \dots, a_n$, and the matroid represented by $A$ over the field $GF(3)$ by $M[A]$. All the computations of the proof are meant over the field $GF(3)$.

Oxley~\cite{oxley1987characterization} (see also \cite[page 658]{oxley2011matroid}) gives the following $GF(3)$-representation of $S(5,6,12)$:
\begin{equation*}
S \coloneqq \left[\begin{array}{*{6}r|*{6}r}
    1 & ~~0 & ~~0 & ~~0 & ~~0 & ~~0 &    0 &  1 & 1  &  1 &  1 & 1 \\
    0 & 1 & 0 & 0 & 0 & 0 &    1 &  0 & 1  & -1 & -1 & 1 \\
    0 & 0 & 1 & 0 & 0 & 0 &    1 &  1 & 0  &  1 & -1 & -1 \\
    0 & 0 & 0 & 1 & 0 & 0 &    1 & -1 & 1  &  0 &  1 & -1 \\
    0 & 0 & 0 & 0 & 1 & 0 &    1 & -1 & -1 &  1 &  0 & 1 \\
    0 & 0 & 0 & 0 & 0 & 1 &    1 &  1 & -1 & -1 &  1 & 0
\end{array}\right].
\end{equation*}
A matrix $T$ representing the matroid $M[S] / s_6 \backslash s_{12}$ 
%(obtained from $M[S]$ by contracting $s_6$ and deleting $s_{12}$) 
can be obtained by deleting the 6th and 12th columns and the 6th row of $S$ (see \cite[Proposition 3.2.6]{oxley2011matroid}):
\begin{equation*}
T \coloneqq \left[
\begin{array}{*{5}r|*{5}r}
1 & ~~0 & ~~0 & ~~0 & ~~0   & 0 &  1 &  1 &  1 &  1 \\
0 & 1 & 0 & 0 & 0   & 1 &  0 &  1 & -1 & -1 \\
0 & 0 & 1 & 0 & 0   & 1 &  1 &  0 &  1 & -1 \\
0 & 0 & 0 & 1 & 0   & 1 & -1 &  1 &  0 &  1 \\
0 & 0 & 0 & 0 & 1   & 1 & -1 & -1 &  1 &  0  
\end{array}
\right].
\end{equation*}
However, it will be more convenient to work with the following matrix:
\begin{equation*}
T' \coloneqq \left[
    \begin{array}{*{5}r|*{5}r}
        1 & ~~0 & ~~0 & ~~0 & ~~0 &   1 &  1 &  1 &  1 & ~~1 \\
        0 & 1 & 0 & 0 & 0 &   0 & -1 & -1 &  1 & 1 \\
        0 & 0 & 1 & 0 & 0 &  -1 &  0 &  1 & -1 & 1 \\
        0 & 0 & 0 & 1 & 0 &   1 & -1 &  0 & -1 & 1 \\
        0 & 0 & 0 & 0 & 1 &  -1 &  1 & -1 &  0 & 1 \\
    \end{array}
\right].
\end{equation*}
The matroids $M[T]$ and $M[T']$ are isomorphic. Indeed, multiplying $T$ by any $5\times 5$ invertible matrix from the left does not change $M[T]$. In particular, the matroid represented by 
\[\left[\begin{array}{*{5}r} t'_4 & t'_2 & t'_1 & t'_3 & t'_6 \end{array}\right] \cdot T = \left[\begin{array}{*{10}r} t'_4 & t'_2 & t'_1 & t'_3 & t'_6 & -t'_7 & t'_5 & -t'_8 & -t'_{10} & -t'_9 \end{array}\right] \] 
is $M[T]$, while it is isomorphic to $M[T']$. Finally, it is not difficult to check that $M[T']/t'_5 \backslash t'_{10}$ is represented by 
\begin{equation*}
P \coloneqq \left[
    \begin{array}{*{4}r |*{4}r}
    1 & ~~0 & ~~0 & ~~0 &   1 &  1 &  1 & 1 \\
    0 & 1 & 0 & 0 &   0 & -1 & -1 & 1 \\
    0 & 0 & 1 & 0 &  -1 &  0 &  1 & -1 \\
    0 & 0 & 0 & 1 &   1 & -1 &  0 & -1
    \end{array}
\right].
\end{equation*}
Besides these matrix representations, we rely on several properties of the matroid $S(5,6,12)$ from \cite[page 658]{oxley2011matroid}. 

We turn to the proof of the lemma, that is, we show that \ref{prop:sbro} holds for $S(5,6,12)$. This property is equivalent to the following: for each pair $A, B$ of bases of $M[S]$ there exists a pair $A', B'$ of disjoint bases of the matroid $N\coloneqq M/(A\cap B)\backslash (E\setminus (A\cap B))$, and a bijection $\varphi\colon A' \to B'$ such that $(A' \setminus X) \cup \varphi(X)$ is a basis of $N$ for every $X \subseteq A$. We distinguish several cases based on the size of $A \cap B$.

If $|A \cap B| \ge 3$, then $N$ has rank at most three. In this case $N$ is strongly base orderable, as $S(5,6,12)$ contains $M(K_4)$ as a minor, and $M(K_4)$-minor-free matroids of rank at most three are strongly base orderable.

If $|A \cap B| = 2$, then we use the fact that the matroid obtained from $S(5,6,12)$ by deleting two elements and contracting two elements is unique up to isomorphism\footnote{Note that the matroid in question is the matroid denoted by $P_8$ in \cite{oxley2011matroid}.}, which follows from the automorphism group of $S(5,6,12)$ being 5-transitive.  Therefore, $N$ is isomorphic to $M[P]$. Setting $A_1 \coloneqq \{p_1, p_2, p_3, p_4\}$, $B_1 \coloneqq \{p_5, p_6, p_7, p_8\}$ and $\varphi_1\colon A_1 \to B_1$, $\varphi_1(p_i) = p_{i+4}$ for $i=1,\dots,4$, it is not difficult to check that $(A_1 \setminus X) \cup \varphi_1(X)$ is a basis of $M[P]$ for every $X \subseteq A_1$.

If $|A \cap B| = 1$, then $N$ is isomorphic to $M[T']$ by a similar reasoning as in the previous case. Let $A_2 \coloneqq \{t'_1, \dots, t'_5\}$, $B_2 \coloneqq \{t'_6, \dots, t'_{10}\}$ and $\varphi_2 \colon A_2 \to B_2, \varphi_2(t'_i) = t'_{i+5}$ for $i=1,\dots, 5$. We claim that $(A_2 \setminus X) \cup \varphi_2(X)$ is a basis of $M[T']$ for every $X \subseteq A_2$. If $t'_5 \not \in X$, then this follows from $((A_2-t'_5) \setminus Z) \cup \varphi_2(Z)$ being a basis of $M[T'] / t'_5 \backslash t'_{10}$ for every $Z \subseteq A_2-t'_5$, which we have already seen in the previous case. If $t'_5 \in X$, then we need to show that $((A_2-t'_5) \setminus Z) \cup \varphi_2(Z)$ is a basis of $M[T'] / t'_{10} \backslash t'_5$ for every $Z \subseteq A'-t'_5$. We obtain the following matrix $P'$ representing $M[T'] / t'_{10} \backslash t'_5$ from $T'$ by pivoting on the 10th element of the 5th row, deleting the 5th and 10th columns and deleting the 5th row:
\begin{equation*}
P' \coloneqq \left[
\begin{array}{*{4}r|*{4}r}
 1 & ~~0 & ~~0 & ~~0 &   -1 &  0 & -1 &  1 \\
 0 & 1 & 0 & 0 &    1 &  1 &  0 &  1 \\
 0 & 0 & 1 & 0 &    0 & -1 & -1 & -1 \\
 0 & 0 & 0 & 1 &   -1 &  1 &  1 & -1 
\end{array}
\right].
\end{equation*}
We need to show that for $A_3 \coloneqq \{p'_1, p'_2, p'_3, p'_4\}$, $B_3 \coloneqq \{p'_5, p'_6, p'_7, p'_8\}$ and $\varphi_3 \colon A_3 \to B_3$, $\varphi_3(p'_i) = \varphi_3(p'_{i+4})$ for $i=1,\dots,4$, the set $(A_3 \setminus Z) \cup \varphi_3(Z)$ is a basis of $M[P']$ for every $Z \subseteq A_3$. As
\begin{equation*}
\left[\begin{array}{*{4}r} p_5 & p_6 & p_3 & -p_8 \end{array}\right] \cdot P' =
\left[\begin{array}{*{8}r} p_5 & p_6 & p_3 & -p_8 &   p_1 &  p_2 & p_7 & -p_4\end{array}\right]
\end{equation*}
represents the same matroid as $P'$, the function $f \colon  (p'_1,\dots, p'_8) \mapsto (p_5,p_6,p_3,p_8,p_1,p_2,p_7,p_4)$ is an isomorphism from $M[P']$ to $M[P]$. Since $f(\{p'_i,p'_{i+4}\}) = \{p_i,p_{i+4}\}$ for $i=1,\dots,4$, $f$ maps the family $\{(A_3 \setminus Z) \cup \varphi_3(Z) \mid Z \subseteq A_3\}$ to $\{(A_1 \setminus X) \cup \varphi_1(X) \mid X \subseteq A_1\}$. As we have already seen that each member of the latter family is a basis of $M[P]$, it follows that each set of the form $(A_3 \setminus Z) \cup \varphi_3(Z)$ is a basis of $M[P']$. 

Finally, consider the case $A \cap B = \emptyset$. As the matroids $M[S] / s_6 \backslash s_{12} = M[T]$ and $M[T']$ are isomorphic, there exists a pair $A_4, B_4$ and a bijection $\varphi_4 \colon A_4 \to B_4$ such that $(A_4 \setminus X) \cup \varphi_4(X)$ is a basis of $M[S]/s_6\backslash s_{12}$ for every $X \subseteq A_4$. Extend $\varphi_4$ to a bijection $\hat \varphi_4 \colon A_4 + s_6 \to B_4 + s_{12}$ by letting $\hat \varphi_4(s_6) = s_{12}$. We claim that $((A_4+s_6)\setminus X) \cup \hat\varphi_4(X)$ is a basis of $M[S]$ for every $X \subseteq A_4+s_6$. If $s_6 \not \in X$, this follows from $(A_4  \setminus X) \cup \varphi_4(X)$ being a basis of $M[S] / s_6 \backslash s_{12}$. If $s_6 \in X$, then we use that $(X-s_6) \cup \varphi_4((A_4+s_6) \setminus X)$ is a basis of $M[S]/s_6 \backslash s_{12}$, hence $X \cup \varphi_4((A_4 + s_6)\setminus X)$ is a basis of $M[S]$. This implies its complement $((A_4+s_6)\setminus X) \cup \varphi_4(X)$ being a basis as well, since $M[S]$ is known to be identically self-dual. 
\end{proof}

\begin{thm}\label{thm:binter}
The matroids $M(K_4)$ and $J$ are excluded minors for $\sbro$. Furthermore, 
\begin{enumerate}[label=(\alph*)]\itemsep0em
    \item a binary matroid is SBRO if and only if it does not contain $M(K_4)$ as a minor, and
    \item a ternary matroid is SBRO if and only if it does not contain $M(K_4)$ or $J$ as a minor.
\end{enumerate}
\end{thm}
\begin{proof}
First we show that neither $M(K_4)$ nor $J$ is SBRO. This follows from Theorem~\ref{thm:sbro} as each $M_1 \in \{M(K_4), J\}$ is 2-coverable and there exists a 2-coverable partition matroid $M_2$ on the same ground set such that $\beta(M_1\cap M_2)=3$. Indeed, for $M_1 = M(K_4)$, let the three partition classes defining $M_2$ be the matchings of size two of $K_4$, see \cite{schrijver2003combinatorial}. For $M_1 = J$, the statement follows by Remark~\ref{rem:antidm}.

$M(K_4)$-minor-free binary matroids are SBRO, as they coincide with the graphic matroids of series-parallel graphs~\cite{brylawski1971combinatorial} which are strongly base orderable, see also \cite{welsh1976matroid}. As $M(K_4)$ is not SBRO, it follows that $M(K_4)$ is the unique binary excluded minor for $\sbro$.

Finally, we show that if $M$ is a ternary excluded minor of $\sbro$ distinct from $M(K_4)$, then $M$ is isomorphic to $J$. As $\sbro$ is closed under direct sums and 2-sums by Theorem~\ref{thm:complete}, it follows that $M$ is 3-connected. Oxley~\cite{oxley1987characterization} showed that a 3-connected $M(K_4)$-minor-free ternary matroid is isomorphic either to the rank-$r$ whirl $\mathcal{W}^r$ for some $r\ge 2$, to $J$, or to a minor of the matroid $S(5,6,12)$. Since $\mathcal{W}^r$ is a transversal matroid, it is strongly base orderable. By Lemma~\ref{lem:issbro}, $S(5,6,12)$ is SBRO. Therefore, $M$ is isomorphic to $J$. As $J$ is not SBRO and does not contain $M(K_4)$ as a minor, it follows that $J$ is the unique ternary excluded minor for $\sbro$ apart from $M(K_4)$.
\end{proof}

As $M(K_4)$ is the only binary, and $M(K_4)$ and $J$ are the only ternary excluded minors for base orderability \cite{oxley1987characterization}, Theorem~\ref{thm:binter} immediately implies the following.

\begin{cor} \label{cor:ternary} 
A binary or ternary matroid is SBRO if and only if it is base orderable.
\end{cor}

Consider a matroid $M=(E,\cC)$ on ground set $E$ with set of circuits $\cC$. Furthermore, assume that $G=(E,F)$ is a graph with vertex set $E$ and edge set $F$. For a subset $X\subseteq E$, let $\cC[X]$ denote the set of circuits of $M$ that lie in $X$, that is, $\cC[X]\coloneqq\{C\in\cC\mid C\subseteq X\}$, and let $F[X]$ denote the set of edges of $G$ induced by $X$. We say that $G$ \textbf{covers} a subset $\cC'\subseteq \cC$ of circuits if $F[C]\neq\emptyset$ for every $C\in\cC'$. In other words, every stable subset of $E$ in $G$ is such that it contains no circuit from $\cC'$.

Let us now return to the open problem of Davies and McDiarmid on replacing strongly base orderability with base orderability in Theorem~\ref{thm:dm}. Theorem~\ref{thm:sbro} and Corollary~\ref{cor:ternary} together imply an affirmative answer in the special case when the matroids are ternary -- for binary matroids, this was already known as the classes of base orderable binary matroids and strongly base orderable binary matroids coincide. Motivated by this observation, it is natural to ask whether all base orderable matroids are SBRO. Unfortunately, the next example shows that this is not the case.

\begin{figure} \centering
\includegraphics[width=0.25\textwidth]{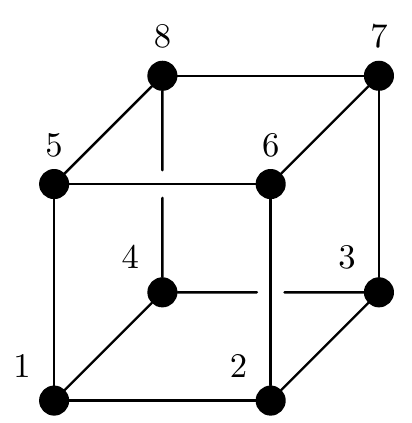}
\caption{The rank-4 matroid $AG(3,2)$ in which each set of size three is independent, and the dependent sets of size four are the six \emph{faces} of the cube, the six \emph{diagonal planes} and the two \emph{twisted planes} $\{1,3,6,8\}$ and $\{2,4,5,7\}$.}
\label{fig:AG32}
\end{figure}

\begin{rem} \label{rem:xten}
%\normalfont 
We construct a rank-5 matroid $X_{10}$ on 10 elements %Oxley $F, J, L, O, P, Q, R, S, T, U, V, W, Z$-et lefoglalta}
that is base orderable but not SBRO. Recall the construction of the the binary affine cube $AG(3,2)$ from \cite[page 645]{oxley2011matroid} using faces, diagonal planes and twisted planes, see Figure~\ref{fig:AG32}. We define $X_{10}$ on the ground set $\{1,2,\dots, 8\} \cup \{a,b\}$ such that each set of size four is independent, and the family of dependent sets of size five is 
\begin{equation*}
\mathcal{H} \coloneqq \left\{F\cup \{a\}\mid  \text{$F$ is a face}\right\} \cup \left\{F\cup \{b\}\mid \text{$F$ is a diagonal plane or a twisted plane}\right\}.
\end{equation*}
By the construction of paving matroids, see e.g.~\cite[Theorem~5.3.5]{frank2011connections}, $X_{10}$ is a paving matroid of rank $5$ as $|H_1 \cap H_2| \le 3$ for each $H_1, H_2 \in \mathcal{H}$, $H_1 \ne H_2$. It is not difficult to check that $X_{10}$ does not contain $M(K_4)$ as a minor, hence it is base orderable since $M(K_4)$-minor-free paving matroids are base orderable~\cite{bonin2016infinite}. 

It remains to show that $X_{10}$ is not SBRO. Assume for contradiction that $X_{10}$ is SBRO and consider a pair $A, B$ of disjoint bases, e.g.\ $A = \{1,2,3,5, a\}$ and $B = \{4,6,7,8,b\}$. By \ref{prop:sbro}, there exists a pair $A', B'$ of disjoint bases and a bijection $\varphi\colon A' \to B'$ such that $(A' \setminus X) \cup \varphi(X)$ is a basis for $X \subseteq A'$, or equivalently, $P \coloneqq \{e\varphi(e) \mid e \in A'\}$ covers each circuit of $X_{10}$. Then we define a bijection $\phiext\colon A'\cup B'\to A'\cup B'$ by setting $\phiext(z)\coloneqq \varphi(z)$ if $z\in A'$ and $\varphi^{-1}(z)$ otherwise.

Since $P$ covers both $\{1,3,6,8,b\}$ and $\{2,4,5,7,b\}$, at least one of the twisted planes $\{1,3,6,8\}$ and $\{2,4,5,7\}$ contains an edge of $P$. We may assume by symmetry that $13 \in P$. Consider the case first when $\phiext(5) \in \{2,4,6,7,8\}$.  Observe that $P' \coloneqq P \setminus \{13,5\phiext(5)\}$ covers each member of the family \[\cH' \coloneqq \{H\setminus \{1,3,5,\phiext(5)\} \mid H' \in \cH, |H' \cap \{1,3\}| = |H'\cap \{5, \phiext(5)\}| = 1\}.\]
We claim that $\cH'$ is the family of circuits of a matroid isomorphic to $M(K_4)$. Indeed,
\[\cH' = \begin{cases} \{\{4,8,a\}, \{4,6,b\}, \{6,7,a\}, \{7,8,b\}\} & \text{if $\phiext(5)=2$,} \\ \{\{2,7,a\},\{2,8,b\}, \{4,7,b\}, \{4,8,a\}\} & \text{if $\phiext(5) = 6$,} \\ \{\{2,6,a\}, \{2,8,b\}, \{4,6,b\}, \{4,8,a\}\} & \text{if $\phiext(5) = 7$,}\end{cases}\]
and the cases $\phiext(5) = 4$ and $\phiext(5) = 8$ are symmetric to the cases $\phiext(5) = 2$ and $\phiext(5) = 6$, respectively. This is a contradiction as the circuits of $M(K_4)$ cannot be covered by a matching.

We proved that $\phiext(5) \in \{a, b\}$ and thus by symmetry $\phiext(7) \in \{a, b\}$ holds as well. We may assume that $\phiext(5) = a$ and $\phiext(7) = b$. Since $P$ covers each of $\{3,4,7,8,a\}$, $\{2,3,6,7,a\}$, $\{3,4,5,6,b\}$ and $\{2,3,5,8,b\}$, the two edges of $P \setminus \{13, 5a, 7b\}$ cover each of $\{4,8\}$, $\{2,6\}$, $\{4,6\}$ and $\{2,8\}$ which is not possible, leading to a contradiction.
\end{rem}

Figure~\ref{fig:classes} summarizes the relation of the matroid classes discussed in the current section. Here $M_\alpha$ denotes the matroid described in \cite[Section 8]{bonin2016infinite} with parameters $r=5$, $|A|=|D|=1$, and $|B|=|C|=|E|=|F| = 2$. That is, let $A=\{a\}$, $B=\{b_1,b_2\}$, $C=\{c_1,c_2\}$, $D=\{d\}$, $E=\{e_1,e_2\}$, and $F=\{f_1,f_2\}$. Then the proper nonempty cyclic flats of $M_\alpha$ are $C\cup B\cup E$, $C\cup A\cup D$, $F\cup E\cup A$ and $F\cup D\cup B$ with ranks $r(C\cup B\cup E)=|C|+|B|=4$, $r(C\cup A\cup D)=|C|+|A|=3$, $r(F\cup E\cup A)=|F|+|E|=4$, and $r(F\cup D\cup B)=|F|+|D|=3$. Let $X$ and $Y$ be arbitrary bases of $M_\alpha$. It was shown in \cite{bonin2016infinite} that $M_\alpha$ is a forbidden minor for strongly base orderability, hence if $X\cap Y\neq\emptyset$ then even \ref{prop:sbo} holds for the pair $X,Y$. On the other hand, if $X$ and $Y$ are disjoint, then it is not difficult to check that the bijection $\varphi\colon (a,b_1,c_1,e_1,f_1) \mapsto (d,b_2,c_2,e_2,f_2)$ satisfies \ref{prop:sbro}.

\begin{figure}[t]
    \centering
    \includegraphics[width=0.5\textwidth]{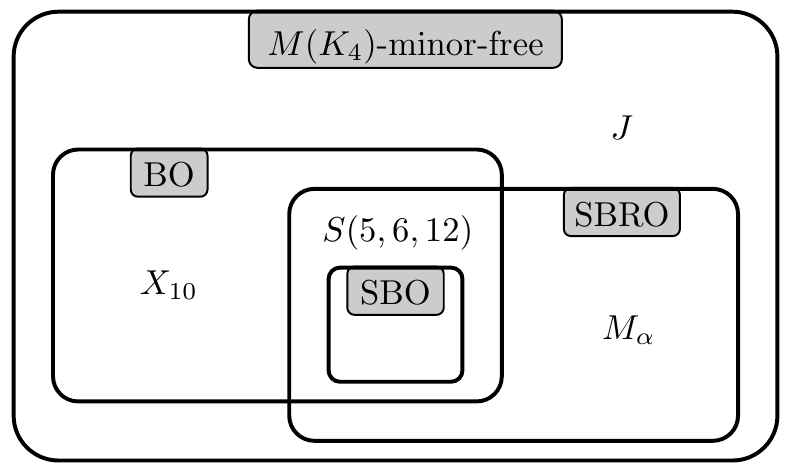}
    \caption{The inclusion relationship between the matroid classes discussed in Section~\ref{sec:sbro}, with examples showing strict containment.}
    \label{fig:classes}
\end{figure}

%%%%%%%%%%%%%%%%
\section{Relaxing the bijection: covering the circuits by a 2-factor}
\label{sec:path}
%%%%%%%%%%%%%%%%

In this section, we propose another relaxation of strongly base orderability. In order to do so, first we give a new interpretation of property \ref{prop:sbo}. Recall that for a matroid $M=(E,\cC)$, a graph $G=(E,F)$ is said to cover $\cC'\subseteq \cC$ if $F[C]\neq\emptyset$ for every $C\in\cC'$. Using this terminology, \ref{prop:sbo} is equivalent to saying that for any pair $A,B$ of bases of a strongly base orderable matroid $M=(E,\cC)$, there exists a graph $G$ consisting of a matching between the elements of $A\setminus B$ and $B\setminus A$ that covers $\cC[A\cup B]$. Similarly, property \ref{prop:sbro} discussed in Section~\ref{sec:sbro} translates into the existence of a matching on $A\triangle B$ that covers $\cC[A\cup B]$. Observe the small but crucial difference between the two definitions: while the former asks for matching edges going between the elements of $A$ and $B$, the latter allows the end vertices of any matching edge to fall in the same set, i.e.\ $A$ or $B$.

We conjecture that a similar statement holds for arbitrary matroids where $G$ is a $2$-regular graph or a path instead of a matching, where a graph is \textbf{$2$-regular} if each vertex has degree exactly two. More precisely, we propose four relaxations of different strengths, and say that a matroid $M=(E,\cC)$ has property
\begin{symenum}\itemsep0em
    \itemsymbol{R}\label{prop:r} if for any pair $A,B$ of bases, there exists a $2$-regular graph on $A\triangle B$ that covers $\cC[A\cup B]$,
    \itemsymbol{R+}\label{prop:rp} if for any pair $A,B$ of bases, there exists a $2$-regular graph that consists of cycles alternating between $A\setminus B$ and $B\setminus A$ and covers $\cC[A\cup B]$ ,
    \itemsymbol{P}\label{prop:p} if for any pair $A,B$ of bases, there exists a path on $A\triangle B$ that covers $\cC[A\cup B]$,
    \itemsymbol{P+}\label{prop:pp} if for any pair $A,B$ of bases, there exists a path that alternates between $A\setminus B$ and $B\setminus A$ and covers $\cC[A\cup B]$.
\end{symenum}
In all cases, the condition that the graph in question covers $\cC[A\cup B]$ is equivalent to requiring that the union of $A\cap B$ and any stable set of $A\triangle B$ in the graph is independent in $M$.

Since any path can be extended to a single cycle by adding an edge between its end vertices, property \ref{prop:pp} implies all the others, while property \ref{prop:r} is a special case of any of them. As for properties \ref{prop:rp} and \ref{prop:p}, we could not show any connection between them. As a matching between $A\setminus B$ and $B\setminus A$ can always be extended to an alternating path between them, strongly base orderable matroids satisfy \ref{prop:pp}, and an analogous reasoning shows that SBRO matroids satisfy \ref{prop:p}. 

%%%%%%%%%%%%%%%%
\subsection{Covering fundamental circuits}
\label{sec:bro}
%%%%%%%%%%%%%%%%

Observe that in all of properties~\ref{prop:r} -- \ref{prop:pp}, the number of edges used to cover the circuits in question is bounded by $|A\triangle B|$. At this point, it is not even clear why those circuits could be covered by a small number of graph edges. As a first step towards understanding the general case, instead of covering every circuit, we concentrate on covering fundamental circuits only. Given a basis $B$ of a matroid and an element $a$ outside of $B$, we denote by $C_B(a)$ the \textbf{fundamental circuit} of $a$ with respect to $B$, that is, the unique circuit in $B+a$. We extend this notation to sets as well, that is, for a set $X$ disjoint from $B$ we use $\cC_B(X)\coloneqq\{C_B(x)\mid x\in X\}$.

\begin{thm} \label{thm:fundamenta}
Let $A,B$ be bases of a matroid $M$. 
\begin{enumerate}[label=(\alph*)]\itemsep0em
    \item\label{it:funda} There exists a 2-regular graph that consists of cycles alternating between $A\setminus B$ and $B \setminus A$ and covers $\cC_A(B)\cup\cC_B(A)$.
    \item\label{it:fundb} There exists a tree that consists of edges between $A\setminus B$ and $B\setminus A$ and covers $\cC_A(B)\cup \cC_B(A)$.
\end{enumerate}
\end{thm}
\begin{proof}
The basis exchange axiom implies that there exists a bijection $\varphi_A\colon A\setminus B\to B\setminus A$ such that $A-a+\varphi_A(a)$ is a basis for each $a\in A\setminus B$, or equivalently, $a\in C_A(\varphi_A(a))$ (see e.g.\ \cite[Theorem 5.3.4]{frank2011connections}).  Similarly, there exists a bijection $\varphi_B\colon B\setminus A\to A\setminus B$ such that $B-b+\varphi_B(b)$ is a basis for each $b\in B\setminus A$, or equivalently, $b\in C_B(\varphi_B(b))$. Therefore the graph consisting of edges $\{a\varphi_A(a)\mid a\in A\setminus B\}\cup\{b\varphi_B(b)\mid b\in B\setminus A\}$ is a $2$-regular graph that covers $\cC_A(B)\cup\cC_B(A)$, proving \ref{it:funda}.

By the symmetric exchange axiom, there exists a mapping -- not necessarily a bijection -- $\phi_A\colon A\setminus B\to B\setminus A$ such that both $A-a+\phi_A(a)$ and $B+a-\phi_A(a)$ are bases for each $a\in A\setminus B$, or equivalently, the edge $a\phi_A(a)$ covers both $C_A(\phi_A(a))$ and $C_B(a)$. Similarly, there exists a mapping $\phi_B\colon B\setminus A\to A\setminus B$ such that both $B-b+\phi_B(b)$ and $A+b-\phi_B(b)$ are bases for each $b\in B\setminus A$, or equivalently, the edge $b\phi_B(b)$ covers both $C_B(\phi_B(b))$ and $C_A(b)$. Consider the graph consisting of edges $\{a\phi_A(a)\mid a\in A\setminus B\}\cup\{b\phi_B(b)\mid b\in B\setminus A\}$. Note that the graph may not be connected, but each vertex in $A\triangle B$ has degree at least one in it. Hence any maximum forest in this graph covers $\cC_A(B)\cup\cC_B(A)$. As extending the forest to a tree by adding edges connecting the components maintains this property, \ref{it:fundb} follows.
\end{proof}

Theorem~\ref{thm:fundamenta} shows that the fundamental circuits corresponding to the basis pair can be covered by a 2-regular graph. However, the same question remains open when the graph is required to be a path, and this seemingly simple task already appears to be highly non-trivial. The essence of part \ref{it:fundb} of the theorem is that, instead of a path, the covering can always be realized by a tree.

%%%%%%%%%%%%%%%%
\subsection{Graphic matroids, paving matroids and spikes}
\label{sec:solved}
%%%%%%%%%%%%%%%%

We could not identify any matroid for which \ref{prop:pp} would fail, hence we propose the following, probably overly optimistic conjecture.

\begin{conj}\label{conj:pp}
Let $A,B$ be bases of a matroid $M$. Then there exists a path that alternates between $A\setminus B$ and $B\setminus A$ and covers $\cC[A\cup B]$. 
\end{conj}

To validate the conjecture somewhat, we show that the statement holds for graphic matroids, paving matroids, and spikes. 
We start with a technical lemma. 

\begin{lem}\label{lem:disj}
Let $\cM$ be a minor-closed class of matroids and ${(X)\in\{\ref{prop:r},\ref{prop:rp},\ref{prop:p},\ref{prop:pp}\}}$. To verify that $(X)$ holds for each $M\in\cM$, it suffices to show that the property holds when the ground set is the disjoint union of $A$ and $B$.
\end{lem}
\begin{proof}
Let $A$ and $B$ be bases of the matroid $M=(E,\cC)$ where $M\in\cM$. If $A\cap B\neq\emptyset$ or $A\cup B\neq E$, then define $A'\coloneqq A\setminus B$, $B'\coloneqq B\setminus A$, $E'\coloneqq A'\cup B'$, and let $M'=(E',\cC')$ denote the matroid obtained from $M$ by deleting $E\setminus (A\cup B)$ and contracting $A\cap B$. Note that $A'$ and $B'$ are disjoint bases of $M'$ whose union is $E'$. Furthermore, for any circuit $C\in\cC[A\cup B]$, there exists a circuit $C'\in\cC'[A'\cup B']$ such that $C'\subseteq C$. Therefore, any graph proving $(X)$ for the pair $A',B'$ in $M'$ also proves $(X)$ for the pair $A,B$ in $M$.   
\end{proof}

With the help of Lemma~\ref{lem:disj}, we first verify the graphic case. It may be confusing that \ref{prop:pp} states the existence of a certain path and we are working with graphs, so let us emphasize that the path is defined on the elements of the ground set, i.e.\ the edges of the underlying graph, and it does not appear in the graph itself in any sense. 
%For a vertex $v$ of a graph $G$, we denote the graph obtained by the deletion of $v$ and the edges incident to it by $G-v$.
%For a graph $G$ and a subset $X$ of edges, we use the same notation $G/X$ and $G\bs X$ for the contraction and deletion of $X$, respectively, in the graph sense, as these operations correspond to the contraction and deletion of $X$ in the graphic matroid $M(G)$. Furthermore, 

\begin{thm} \label{thm:graphic} 
Graphic matroids have property \ref{prop:pp}.
\end{thm}
\begin{proof}
Let $G=(V,E)$ be a graph with graphic matroid $M(G)$. We prove the theorem by induction on $|E|$. Note that \ref{prop:pp} clearly holds when $E=\emptyset$. Take a pair of bases $A,B$. As the class of graphic matroids is closed under taking minors, Lemma~\ref{lem:disj} applies, hence we may assume that $E$ is the disjoint union of $A$ and $B$. Furthermore, we may assume that $G$ is connected as otherwise we can pick a vertex in each connected component and identify those, thus obtaining a connected graph with the same graphic matroid as the original one. 

By the above, we have $|E|=|A|+|B| = 2(|V|-1)$, hence $G$ has a vertex $v$ of degree at most three. If $v$ has degree two, then $G$ has an edge $a \in A$ and an edge $b \in B$ incident to $v$. Since $G' \coloneqq G-v$ is the union of disjoint spanning trees $A-a$ and $B-b$, there exists a path $P'$ covering the circuits of $M(G')$ by the induction hypothesis. As every cycle of $G$ passing through $v$ uses the edges $a$ and $b$, adding an extra edge to $P'\cup \{ab\}$ between $a$ and the endpoint of $P'$ in $B$ results in an alternating path between $A$ and $B$ covering the circuits of $M(G)$.

If $v$ has degree three, then we may assume that edges $a_1, a_2 \in A$ and $b \in B$ are incident to $v$. Consider the graph $G'$ obtained by contracting $a_2$ and deleting $b$, let $A'$ denote the spanning tree of $G'$ obtained from $A$ by contracting $a_2$, and define $B'\coloneqq B-b$. By the induction hypothesis, there exists a path $P'$ alternating between $A'$ and $B'$ that covers the circuits of $M(G')$. Let $c$ be a neighbour of $a_1$ in $P'$ and consider the path $P \coloneqq (P'-a_1c)\cup\{a_1b,ba_2,a_2c\}$. We claim that $P$ covers every circuit $C$ of $M(G)$. If $C$ corresponds to a cycle of $G$ that does not pass through $v$, then it is also a circuit of $M(G')$, and therefore it is covered by $P'-a_1c$. If $b \in C$, then either $a_1 \in C$ or $a_2 \in C$, hence $C$ is covered either by $a_1b$ or $ba_2$. The only remaining case is when $a_1, a_2 \in C$. Let $C'$ denote the cycle of $G'$ obtained by contracting $a_2$ in $C$. Then $C'$ is covered by an edge of $P'$. If $c \not\in C$, then this is an edge of $P'-a_1c$ as well, otherwise $C$ is covered by the edge $a_2c$. This proves that $P$ covers all circuits of $M(G)$.
\end{proof}

\begin{rem} \label{rem:graphic}
\normalfont
The fact that the path in \ref{prop:pp} is defined on the edges of the graph might be confusing, hence let us rephrase Theorem~\ref{thm:graphic} using graph terminology as it might be of independent combinatorial interest. The theorem is equivalent to the following: For any graph that is the union of two spanning trees $A$ and $B$, its edges can be ordered in such a way that the elements of $A$ and $B$ appear alternately, and every cycle of $G$ contains two consecutive elements. 
\end{rem}

A rank-$r$ matroid is called \textbf{paving} if each set of size at most $r-1$ is independent, or in other words, each circuit has size $r$ or $r+1$. 

\begin{thm} \label{thm:paving} 
Paving matroids have property \ref{prop:pp}.
\end{thm}
\begin{proof} 
We prove the theorem by induction on $|E|$. Note that \ref{prop:pp} clearly holds when $E=\emptyset$. Take a pair of bases $A,B$. As the class of graphic matroids is closed under taking minors, Lemma~\ref{lem:disj} applies, hence we may assume that $E$ is the disjoint union of $A$ and $B$.

Let $a_r \in A$ and $b_r \in B$ such that $A-a_r+b_r$ is a basis. Then $A-a_r$ and $B-b_r$ are bases of the paving matroid $M'$ obtained from $M$ by contracting $b_r$ and deleting $a_r$. By the induction hypothesis, there exist orderings $a_1,\dots, a_{r-1}$ of $A-a_r$ and $b_1,\dots, b_{r-1}$ of $B-b_r$ such that the path $P' = a_1,b_1,\dots, a_{r-1}, b_{r-1}$ covers the circuits of $M'$. We claim that the path $P = a_1,b_1,\dots, a_r, b_r$ covers each circuit of $M$. Since $M$ is paving, we only need to consider circuits $C$ of size $r$ as circuits of size $r+1$ are clearly covered. If $a_r \not \in C$, then $C-b_r$ contains a circuit of $M'$ which is covered by $P'$. If $a_r \in C$, then $C$ is covered by $P$ as $A$ is the only stable set of size $r$ containing $a_r$. Therefore, $P$ covers every circuit of $M$.
\end{proof}

A rank-$r$ \textbf{spike} with \textbf{tip} $t$ and \textbf{legs} $\{t,x_1,y_1\}, \dots, \{t,x_r,y_r\}$ is a rank-$r$ matroid on the ground set $E=\{t,x_1,y_1,\dots, x_r,y_r\}$ such that each leg is a circuit and, for $1 \le k \le r-1$, the union of any $k$ legs has rank $k+1$. Equivalently, the family of the non-spanning circuits of a spike consists of the legs, all sets of the form $\{x_i, y_i, x_j, y_j\}$ with $1 \le i < j \le r$, and a subset $\cH \subseteq \{Z\subseteq E \mid |Z|=r,\ |Z\cap\{x_i,y_i\}|=1\ \text{for $i=1,\dots,r$}\}$ such that the intersection of any two members of $\cH$ has size at most $r-2$.

\begin{thm} \label{thm:spike} Spikes have property \ref{prop:pp}.
\end{thm}
\begin{proof}
Let $M$ be a spike defined as above. Observe that a set $Z \subseteq E$ is a basis if and only if one of the following holds.
\begin{enumerate}\itemsep0em
\item $t \in Z$, $\{x_\ell, y_\ell\} \cap Z = \emptyset$ for an index $1 \le \ell \le r$, and $|\{x_i,y_i\} \cap Z| = 1$ for  all $i \in \{1,\dots, r\}\setminus\{\ell\}$.
\item $t \not \in Z$ and there exist distinct indices $k$ and $\ell$ such that $\{x_k,y_k\} \subseteq Z$, $\{x_\ell, y_\ell\} \cap Z = \emptyset$ and $|\{x_i,y_i\} \cap Z| = 1$ for all $i \in \{1,\dots, r\} \setminus \{k,\ell\}$.
\item $t \not \in Z$, $|\{x_i,y_i\} \cap Z| = 1$ for all $1 \le i \le r$, and $Z \not \in \cH$.
\end{enumerate}
We show by induction on $r$ that each pair $A, B$ of bases of $M$ has the property required by \ref{prop:pp}. The statement clearly holds for $r\leq 2$, hence we assume that $r\geq 3$.

First consider the case when $A \cap B \ne \emptyset$. If $t \in A \cap B$, then the statement follows by applying \ref{prop:pp} to the pair $A-t, B-t$ of bases of the matroid $M/t$. Note that $M/t$ is a laminar, thus strongly base orderable matroid with family of bases $\{Z\subseteq E-t \mid |Z| = r-1, |\{x_i,y_i\} \cap Z| = 1 \text{ for all $1 \le i \le r$}\}$. If $t \not \in A \cap B$, then we may assume that $x_r \in A \cap B$. If $t \not \in A \cup B$, then the statement follows by applying the induction hypothesis to the pair of bases $A-x_r, B-x_r$ of the matroid $M/x_r\backslash t$ which is a spike with tip $y_r$ and legs $\{x_1,y_1\}, \dots, \{x_{r-1}, y_{r-1}\}$. Similarly, if $y_r \not \in A \cup B$, then we can apply the induction hypothesis to the pair of bases $A-x_r, B-x_r$ of $M/x_r\backslash y_r$ which is a spike with tip $t$ and legs $\{x_1,y_1\}, \dots, \{x_{r-1}, y_{r-1}\}$. It remains to consider the case $t \in A \triangle B$, $x_r \in A\cap B$ and $y_r \in A \cup B$, when we may assume that $t \in A \setminus B$ and $y_r \in B \setminus A$. As $z$ and $y_r$ are parallel in $M/x_r$, the matroids $M/\{x_r,t\} \backslash \{y_r\}$ and $M/\{y_r,x_r\}\backslash\{t\}$ are the same. Let us denote this matroid by $M'$ -- note that $M'$ is strongly base orderable as it is the minor of the laminar matroid $M/t$. Let $P'$ denote a path obtained by applying \ref{prop:pp} to the pair of bases $A-x_r-t$ and $B-x_r-y_r$ of $M'$. Then $P'+y_rt$
satisfies the requirements of \ref{prop:pp} for the basis pair $A, B$ in $M$, as the union of any independent set of $M'$ and either $\{x_r,t\}$ or $\{x_r,y_r\}$ is independent in $M$.

In the remaining cases $A \cap B = \emptyset$. Suppose first that $t \in A \cup B$. We may assume that $A = \{t,x_1,\dots,x_{r-1}\}$, and either $B=\{y_2,\dots,y_r, x_r\}$ or $B=\{y_1,\dots,y_r\}$. If $B=\{y_2,\dots, y_r,x_r\}$, then the path $P=x_2,y_2,\dots, x_{r-1}, y_{r-1}, t, x_r, x_1, y_r$ alternating between $A$ and $B$ covers each circuit of $M\backslash y_1$. Indeed, $x_iy_i \in P$ for each $2\le i \le r-1$, it covers the leg $\{t,x_r,y_r\}$ and it has no stable set $Z$ with $|\{x_1,y_1\} \cap Z| = \dots = |\{x_r,y_r\} \cap Z| = 1$. If $B=\{y_1,\dots, y_r\}$, then the path $P=t,y_r,x_1,y_1,\dots, x_{r-1},y_{r-1}$ covers each circuit of $M \backslash x_r$. Indeed, $x_iy_i \in P$ for each $1 \le i \le r-1$ and the only stable set $Z$ of $P$ with $|\{x_1,y_1\}\cap Z| = \dots = |\{x_r,y_r\}\cap Z| = 1$ is the basis $Z=\{y_1,\dots, y_r\} = B$.

It remains to consider cases with $A \cap B = \emptyset$ and $A \cup B = E-t$. If $A$ contains a pair $\{x_i,y_i\}$, then we may assume that $A=\{x_1,y_1,x_3,x_4,\dots, x_r\}$ and $B=\{x_2,y_2,y_3,y_4,\dots,y_r\}$. Consider the paths $P_1 = x_1,y_2,y_1,x_2,x_3,y_3,\dots,x_r,y_r$ and $P_2 = x_1,x_2,y_1,y_2,x_3,y_3,\dots, x_r,y_r$. Each of $P_1$ and $P_2$ covers the circuits of the form $\{x_i,y_i,x_j,y_j\}$ for $1 \le i < j \le r$, since both of them contains the edges $x_3y_3, \dots, x_ry_r$ and covers $\{x_1,y_1,x_2,y_2\}$. The only set $Z$ with $|\{x_1,y_1\} \cap Z| = \dots = |\{x_r,y_r\} \cap Z| = 1$ which is a stable set of $P_1$ or $P_2$ is $Z_1 = \{x_1,x_2,y_3,y_4,\dots, y_r\}$ and $Z_2 = \{x_1,y_2,y_3,\dots, y_r\}$, respectively. Since $|Z_1 \cap Z_2| = r-1$, at least one of $Z_1$ and $Z_2$ is not in $\cH$ and thus at least one of $P_1$ and $P_2$ covers every circuit of $M\backslash t$.

In the only remaining case we may assume that $A=\{x_1,\dots, x_r\}$ and $B=\{y_1,\dots, y_r\}$. We claim that the path $P=y_r,x_1,y_1,\dots,x_{r-1},y_{r-1},x_r$ covers each circuit of $M\backslash t$. Indeed, $x_iy_i \in P$ for each $1 \le i \le r-1$ and the only stable sets $Z$ of $P$ with $|\{x_1,y_1\}\cap Z| = \dots = |\{x_r,y_r\}\cap Z| = 1$ are the bases $Z=\{x_1,\dots, x_r\} = A$ and $Z=\{y_1,\dots, y_r\} = B$.
\end{proof}

%%%%%%%%%%%%%%%%
\subsection{Applications for covering problems}
\label{sec:cor}
%%%%%%%%%%%%%%%%

Motivated by a conjecture of Du, Hsu and Hwang~\cite{du1993hamiltonian}, Erd\H{o}s \cite{erdos1990some} popularized the so-called \emph{cycle plus triangles problem} which asked whether every 4-regular graph that is the edge-disjoint union of a Hamiltonian cycle and pairwise vertex-disjoint triangles is 3-colorable. Fleischner and Stiebitz~\cite{fleischner1992solution} answered this question in the affirmative. In the past decades, their work was followed by a series of papers that studied related problems, summarized below.

\begin{thm}{\normalfont(Fleischner and Stiebitz, McDonald and Puleo, Haxell)} \label{thm:chromatic} 
\begin{enumerate}[label=(\alph*)]\itemsep0em 
\item If a graph $G$ is the edge-disjoint union of a Hamiltonian cycle and some pairwise vertex-disjoint triangles, then $G$ is 3-colorable. \label{it:triangles}
\item If $k \ge 4$ and a graph $G$ is the edge-disjoint union of a Hamiltonian cycle and some pairwise vertex-disjoint complete subgraphs each on at most $k$ vertices, then $G$ is $k$-colorable. \label{it:cliques}
\item If $k \ge 4$ and a graph $G$ is the edge-disjoint union of a 2-regular bipartite graph and some pairwise disjoint cliques each on at most $k$ vertices, then $G$ is $k$-colorable. \label{it:even2factor}
\item If $k \ge 5$ and a graph $G$ is the edge-disjoint union of a 2-regular graph and some pairwise vertex-disjoint complete subgraphs each on at most $k$ vertices, then $G$ is $k$-colorable. \label{it:color_2factor}
\end{enumerate}
\end{thm}

Statement \ref{it:cliques} was proven by Fleischner and Stiebitz~\cite{fleischner1997remarks}. McDonald and Puleo~\cite{mcdonald2022strong} verified for $k \ge 4$ that if a graph is decomposable into cliques on exactly $k$ vertices and a 2-regular graph with at most one odd cycle of length exceeding three, then it is $k$-colorable. This implies statement~\ref{it:even2factor}, as we can extend the cliques having less than $k$ vertices and the 2-regular bipartite graph by adding extra vertices and edges to ensure that each clique has exactly $k$ vertices. Finally, statement \ref{it:color_2factor} follows from the result of Haxell~\cite{haxell2004strong} that a graph is $k$-colorable if it is decomposable into cliques on $k$ vertices and a graph $H$ such that $k \ge 3\Delta(H)-1$. Note that the complete graph on four vertices shows that statements \ref{it:cliques} and \ref{it:even2factor} do not hold for $k=3$, while it is open whether \ref{it:color_2factor} holds for $k=4$. 

\begin{thm}\label{thm:partition}
Let $M_1=(E,\cI_1)$ be a 2-coverable matroid and $M_2=(E,\cI_2)$ be a $k$-coverable partition matroid.
\begin{enumerate}[label=(\alph*)]\itemsep0em 
    \item If $k \ge 3$ and $M_1$ satisfies \ref{prop:p}, then $\beta(M_1\cap M_2) \le k$. \label{it:partition_p}
    \item If $k \ge 4$ and $M_1$ satisfies \ref{prop:rp}, then $\beta(M_1\cap M_2) \le k$. \label{it:partition_rp}
    \item If $k\ge 5$ and $M_1$ satisfies \ref{prop:r}, then $\beta(M_1\cap M_2) \le k$. \label{it:partition_r}
\end{enumerate}
\end{thm}
\begin{proof}
To prove \ref{it:partition_p}, let $E=A\cup B$ be a decomposition of $M_1$ into two bases and $P$ be a path on $A\triangle B$ covering the circuits of $M_1$. Let $E_1,\dots,E_q$ denote the partition classes of $M_2$ and let $Q$ denote the graph obtained by taking the union of complete graphs on $E_i$ for $i=1,\dots, q$. We claim that the graph $G\coloneqq P\cup Q$ is $k$-colorable. Indeed, if $k=3$, then $G$ is a subgraph of a graph appearing in Theorem~\ref{thm:chromatic}\ref{it:triangles}, while for $k\ge 4$ we can apply Theorem~\ref{thm:chromatic}\ref{it:cliques} after adding an extra edge to $G$ between the end vertices of $P$. As each stable set of $P$ is independent in $M_1$ and each stable set of $Q$ is independent in $M_2$, the $k$-colorability of $G$ implies that $\beta(M_1\cap M_2) \le k$. 

Parts~\ref{it:partition_rp} and \ref{it:partition_r} can be proved analogously using Theorem~\ref{thm:chromatic}\ref{it:even2factor} and \ref{it:color_2factor}, respectively.
\end{proof}

One of the main consequences of the results discussed so far is that we confirm Conjecture~\ref{conj:ab} for instances that were not settled before.

\begin{cor} 
Conjecture~\ref{conj:ab} holds if $M_1$ is either a graphic matroid, a paving matroid or a spike with $\beta(M_1)=2$, and $M_2$ is a partition matroid.
\end{cor}
\begin{proof}
The conjecture was settled for $\beta(M_2)\leq 2$ in \cite{aharoni2012edge}. For $\beta(M_2)\geq 3$, the statement follows by combining Theorems~\ref{thm:graphic}\,--\,\ref{thm:spike} and Theorem~\ref{thm:partition}\ref{it:partition_p}.
\end{proof}

Kotlar and Ziv~\cite{kotlar2005partitioning} defined an element $e$ of a matroid $M$ to be \textbf{$(k+1)$-spanned} if there exist $k+1$ pairwise disjoint sets such that $e$ is spanned by each of them in $M$. This condition is equivalent to the existence of $k$ pairwise disjoint sets not containing $e$ but spanning it, as one of the sets can be chosen to be $\{e\}$. If a matroid does not contain any $(k+1)$-spanned element, then it is not difficult to show that it is $k$-coverable. Kotlar and Ziv conjectured that if no element is $(k+1)$-spanned in either $M_1$ or $M_2$, then $\beta(M_1\cap M_2) \le k$. They verified the conjecture if $k=2$ or the ground set decomposes into $k$ bases in each of $M_1$ and $M_2$. Next we show how the absence of $(k+1)$-spanned elements can be combined with \ref{prop:r} or \ref{prop:p}. 

\begin{thm} 
Let $M_1=(E,\cI_1)$ be a 2-coverable matroid and $M_2=(E,\cI_2)$ be a matroid with no $(k+1)$-spanned elements.
\begin{enumerate}[label=(\alph*)]\itemsep0em
\item If $M_1$ satisfies \ref{prop:p}, then $\beta(M_1\cap M_2) \le k+1$. \label{it:ab_path}
\item If $M_1$ satisfies \ref{prop:r}, then $\beta(M_1\cap M_2) \le k+2$. \label{it:ab_2factor}
\end{enumerate}
\end{thm}
\begin{proof}
To prove \ref{it:ab_path}, let $E=\{e_1, \dots, e_n\}$ be such that every stable set of the path defined by the ordering $e_1, \dots, e_n$ is independent in $M_1$. Color the elements in the order $e_1, \dots, e_n$ greedily with positive integers such that an element $e_i$ receives the smallest color $c$ which is distinct from the color of $e_{i-1}$, and the set of elements already having color $c$ does not span $e_i$ in $M_2$. This procedure results in a coloring such that each color class is independent in $M_2$ and form a stable set of the path, thus it is independent in $M_1$ as well. As $M_2$ contains no $(k+1)$-spanned elements, the number of colors used is at most $k+1$. This proves that $\beta(M_1\cap M_2) \le k+1$.

Statement \ref{it:ab_2factor} follows by a similar greedy argument.
\end{proof}

Aharoni and Berger \cite{aharoni2006intersection} defined a graph $G$ to be \textbf{matroidally $k$-colorable} if for every $k$-coverable matroid $M$ on the vertex set of $G$, the ground set can be decomposed into $k$ stable sets of $G$ which are independent in $M$. The following facts are not difficult to show for matroidally $k$-colorable graphs.

\begin{lem} \label{lem:col}
\mbox{}
\begin{enumerate}[label=(\alph*)]\itemsep0em
    \item A subgraph of a matroidally $k$-colorable graph is matroidally $k$-colorable. \label{it:subgraph}
    \item A matroidally $k$-colorable graph is matroidally $(k+1)$-colorable. \label{it:larger}
\end{enumerate}
\end{lem}
\begin{proof}
The proof of \ref{it:subgraph} is straightforward. To prove \ref{it:larger}, let $G$ be a matroidally $k$-colorable graph on vertex set $V$ and $M$ be a $(k+1)$-coverable matroid on $V$. Let $V= I_1 \cup \dots \cup I_{k+1}$ be a partition of $V$ into independent sets of $M$. As the graph $G[I_1 \cup \dots I_k]$ is matroidally $k$-coverable by \ref{it:subgraph} and $M\backslash I_{k+1}$ is a $k$-coverable matroid on its vertex set, there exist stable sets $S_1,\dots, S_k$ of $G$ which are independent in $M\backslash I_{k+1}$ such that $S_1 \cup \dots \cup S_k = I_1 \cup \dots \cup I_k$. Since $G[S_2 \cup \dots \cup S_k \cup I_{k+1}]$ is matroidally $k$-coverable by \ref{it:subgraph} and $M\backslash S_1$ is a $k$-coverable matroid on its vertex set, there exist stable sets $S'_2, \dots, S'_{k+1}$ of $G$ which are independent in $M \backslash S_1$ such that $S'_2 \cup \dots\cup S'_{k+1} = S_2  \cup \dots \cup S_k \cup I_{k+1}$. Then $V=S_1\cup S'_2 \cup \dots \cup S'_{k+1}$ is a partition of $V$ into stable sets of $G$ which are independent in $M$.
\end{proof}

As a generalization of Theorem~\ref{thm:chromatic}\ref{it:triangles}, Aharoni and Berger~\cite{aharoni2006intersection} conjectured that the cycle $C_{3\ell}$ is matroidally 3-colorable for every $\ell\geq 1$. Their conjecture, when combined with Lemma~\ref{lem:col}\ref{it:subgraph}, would imply the following.

\begin{conj}\label{conj:pathcol} 
Every path is matroidally 3-colorable.
\end{conj}

\begin{rem} 
\normalfont
It is reasonable to ask whether there is a connection between Conjectures~\ref{conj:ab} and \ref{conj:pathcol}. It turns out that the latter implies the former when $M_1$ is $2$-coverable matroid satisfying \ref{prop:p}.

Indeed, let $P$ be a path on vertex set $E$ such that every stable set of $P$ is independent in $M_1$,  and let $M_2=(E,\cI_2)$ be an arbitrary $k$-coverable matroid. If $k=2$, then $\beta(M_1\cap M_2) \le 3$ holds by the results of Aharoni, Berger and Ziv~\cite{aharoni2012edge}. If $k \ge 3$ and Conjecture~\ref{conj:pathcol} holds, then $P$ is matroidally $k$-colorable by Lemma~\ref{lem:col}\ref{it:larger}, hence $E$ can be decomposed into $k$ stable sets of $P$ which are independent set in $M_2$. This gives a decomposition of $E$ into $k$ common independent sets of $M_1$ and $M_2$. 

By the above reasoning, Conjectures~\ref{conj:pp} and \ref{conj:pathcol} together would imply Conjecture~\ref{conj:ab} when $M_1$ is $2$-coverable. 
\end{rem}

Finally, let us mention an interesting result on the intersection of $q$ matroids satisfying \ref{prop:r}. Let $M_1,\dots,M_q$ be $2$-coverable matroids over the same ground set. In general, not much is known about the minimum number of common independent sets $\beta(M_1\cap\dots,\cap M_q)$ needed to cover their ground set. An obvious upper bound is the product of their individual covering numbers, that is, $\beta(M_1\cap\dots\cap M_q)\leq\prod_{i=1}^q \beta(M_i)=2^q$. Nevertheless, a much stronger upper bound follows if each $M_i$ satisfies \ref{prop:r}.  

\begin{thm} 
Let $M_1, \dots, M_q$ be 2-coverable matroids satisfying \ref{prop:r}. Then $\beta(M_1 \cap \dots \cap M_q) \le 2q+1$.
\end{thm}
\begin{proof}
Let $G_i$ be a 2-regular graph covering the circuits of $M_i$ for $1 \le i \le q$. Since the graph $G \coloneqq G_1 \cup \dots \cup G_q$ has maximum degree at most $2q$, it is $(2q+1)$-colorable. Since every stable set of $G$ is stable in each $G_i$, this coloring gives a decomposition of the ground set into $2q+1$ common independent sets of the matroids. 
\end{proof}

%%%%%%%%%%%%%%%%
\subsection{Applications for exchange sequences}
\label{sec:gabow}
%%%%%%%%%%%%%%%%

While studying the structure of symmetric exchanges in matroids, Gabow~\cite{gabow1976decomposing} formulated a beautiful conjecture on cyclic orderings of matroids. The question was later raised again by Wiedemann \cite{wiedemann1984cyclic} and by Cordovil and Moreira~\cite{cordovil1993bases}.

\begin{conj} {\normalfont (Gabow)} \label{conj:gabow}
Let $A$ and $B$ be bases of a rank-$r$ matroid $M$. Then there exist orderings $(a_1,\dots,a_r)$ and $(b_1,\dots,b_r)$ of the elements of $A$ and $B$, respectively, such that their concatenation $(a_1,\dots,a_r,b_1,\dots,b_r)$ is a cyclic ordering in which any interval of length $r$ forms a basis.
\end{conj}

An easy reasoning shows that the statement holds for strongly base orderable matroids. Apart from this, the conjecture was settled for graphic matroids~\cite{kajitani1988ordering,cordovil1993bases,wiedemann1984cyclic}, sparse paving matroids~\cite{bonin2013basis}, matroids of rank at most $4$~\cite{kotlar2013serial} and $5$~\cite{kotlar2013circuits}, split matroids~\cite{berczi2022exchange}, and spikes~\cite{berczi2022weighted}. In what follows, we show that \ref{prop:pp} implies a slightly weakened version of the conjecture.

\begin{thm} \label{thm:weakgab}
Let $A$ and $B$ bases of a rank-$r$ matroid $M$ satisfying \ref{prop:pp}. Then there exist orderings $(a_1,\dots,a_r)$ and $(b_1,\dots,b_r)$ of the elements of $A$ and $B$, respectively, such that their concatenation $(a_1,\dots,a_r,b_1,\dots,b_r)$ is a cyclic ordering in which any interval $\{a_i,\dots,a_r,b_1,\dots,b_{i-1}\}$ forms a basis for $i=1,\dots,r$, and any interval $\{b_i,\dots,b_r,a_1,\dots,a_{i-1}\}$ has rank at least $r-1$ for $i=1,\dots,r$.
\end{thm}
\begin{proof}
Let $k = |A\setminus B| = |B \setminus A|$. By property \ref{prop:pp}, there exists a path $P$ that alternates between $A\setminus B$ and $B\setminus A$ such that the union of $A\cap B$ and any stable set in the path is independent in $M$. Let the elements along this path be denoted by $b_1,a_1,\dots,b_k,a_k$, where $a_i\in A\setminus B$ and $b_i\in B\setminus A$ for $i=1,\dots,k$. Furthermore, take two copies of each element in $A\cap B$ and denote those by $a_j$ and $b_j$ for $j=k+1,\dots,r$. We claim that the ordering $(a_1,\dots,a_r,b_1,\dots,b_r)$ satisfies the conditions of the theorem. To see this, observe that, for $1\leq i\leq r$, the interval $\{a_i,\dots,a_r,b_1,\dots,b_{i-1}\}$ is the union of $A\cap B$ and a set that is stable on the path $P$. Similarly, for $1\leq i\leq r$, the interval $\{b_i,\dots,b_r,a_1,\dots,a_{i-1}\}$ is the union of a copy of $A\cap B$ and a set that spans at most a single edge of the path, i.e.\ $a_{i-1}b_i$. However, deleting either $a_{i-1}$ or $b_i$ from the set, we get an independent set, concluding the proof. 
\end{proof}

%\begin{rem}
%\normalfont 
%A matroid is called \textbf{uniformly dense} if $r(E)\cdot |X|\leq r(X)\cdot |E|$ holds for every $X\subseteq E$. Van den Heuvel and Thomassé \cite{van2012cyclic} showed that if $M$ is a uniformly dense matroid with $|E|$ and $r(E)$ being coprimes, then the elements of $E$ have a cyclic ordering such that any $r$ cyclically consecutive elements form a basis of $M$. Their result has an interesting consequence that is usually not emphasized, though it is closely related to Conjecture~\ref{conj:gabow}. Namely, it implies that if the ground set $E$ of a rank-$r$ matroid can be partitioned into two bases, then there exists a cyclic ordering $(e_1,\dots,e_{2r})$ of its elements such that any interval $\{e_i,e_{i+1},\dots,e_{i+r-1}\}$ forms a basis if $i=1,\dots,r$ and has rank at least $r-1$ if $i=r+1,\dots,2r$. \tnote{gondoljuk végig}

%Theorem~\ref{thm:weakgab} shows that \ref{prop:pp} provides a similar relaxation of Gabow's conjecture which is stronger in the sense that the elements of $A$ and $B$, just as in the original conjecture, form intervals.
%\end{rem}

\subsection{Covering circuits with few edges} \label{sec:few}

Each of properties \ref{prop:r}\,--\,\ref{prop:pp}, and hence Conjecture~\ref{conj:pp} states that the circuits of a matroid $M$ of rank $r(M)$ with $\beta(M)=2$ can be covered by $O\left(2\cdot r(M)\right)$ edges. This rightly raises the question of whether a similar statement could hold for higher values of the covering number, e.g.\  whether all circuits of a matroid $M$ can be covered by a small number of edges. The next lower bound only uses the fact that each independent set of $M$ has size at most $r(M)$.

\begin{thm} \label{thm:lower}
Let $M$ be a matroid such that its ground decomposes into disjoint bases. Then every graph covering the circuits of $M$ has at least $\Theta\left(\beta(M)^2 \cdot  r(M)\right)$ edges.
\end{thm}
\begin{proof}
Let $G=(E,F)$ be a graph covering the circuits of $M$ and let $\alpha(G)$ denote the size of the largest stable set of $G$. As each stable set of $G$ is independent in $M$, $\alpha(G) \le r(M)$ holds. It is known (see e.g.\ \cite[Theorem 13.5]{berge1970graphes}) 
that $\alpha(G) \ge \frac{n^2}{n+2m}$ holds for each graph with $n$ vertices and $m$ edges. Therefore, 

\[r(M) \ge \alpha(G) \ge \frac{|E|^2}{|E| + 2|F|} = \frac{\beta(M)^2 \cdot r(M)^2}{\beta(M) \cdot r(M) + 2 |F|},\]
which implies
\[|F| \ge \frac{\beta(M)^2 \cdot r(M)  - \beta(M) \cdot r(M)}{2} = \Theta\left(\beta(M)^2 \cdot r(M)\right).\]
\end{proof}

The lower bound of Theorem~\ref{thm:lower} is consistent with a conjecture that was suggested by the authors and Yamaguchi~\cite{berczi2021list}, stating that any matroid $M$ of rank $r(M)$ has a so-called reduction to a partition matroid with partition classes of size at most $2\cdot \beta(M)$. Indeed, taking a complete graph on the elements of each partition class would result in a graph of at most ${2\cdot \beta(M) \choose 2}\cdot r(M)$ edges that covers every circuit of the matroid. The conjecture was disproved in \cite{abolazimi2021matroid} and independently in \cite{leichter2022impossibility}, and their proofs also show a lower bound of $\Theta\left(\beta(M)^2\cdot\log\beta(M)\cdot r(M)\right)$. 

For any matroid $M$ the size of the ground set is bounded by $\beta(M)\cdot r(M)$, hence $O(\beta(M)^2\cdot r(M)^2)$ edges always suffice to cover all the circuits -- just take a complete graph on $E$. It remains an open problem whether there exists a function $f$ such that for any matroid $M$, the circuits can be covered using $O(f(\beta(M))\cdot r(M))$ edges.

%%%%%%%%%%%%%%%%
\section{Conclusions}
\label{sec:conc}
%%%%%%%%%%%%%%%%

In this paper, we studied relaxations of strongly base orderability of matroids, and showed that the combination of the proposed relaxations with existing results and conjectures leads to interesting observations. Our hopes are that this new approach will serve as a useful tool in matroid optimization problems, especially when working with disjoint common independent sets of two or more matroids. 

Conjecture~\ref{conj:pp} remains an intriguing open problem even for special classes of matroids. While Theorem~\ref{thm:graphic} settles the conjecture for graphic matroids, an analogous result is missing for the co-graphic case. Similarly to Remark~\ref{rem:graphic}, it is worth rephrasing the question using graph terminology. 

\begin{qu}
Given a graph that is the union of two spanning trees $A$ and $B$, is it always possible to order the edges in such a way that the elements of $A$ and $B$ appear alternately, and every cut of $G$ contains two consecutive elements?
\end{qu}

Another interesting question is to characterize the forbidden minors of SBRO matroids. We have seen that $M(K_4)$ and $J$ are forbidden minors, and one can also verify that the matroid $X_{10}$ described in Remark~\ref{rem:xten} is also a forbidden minor for this class. However, we are not aware of any further examples, though the cases of $\sbo$ and $\bo$ matroids suggests that one should expect an infinite number of forbidden minors, see~\cite{bonin2016infinite}.

\begin{qu}
What are the forbidden minors of the class $\sbro$? Is there an infinite sequence of them?
\end{qu}

Finally, let us repeat the question raised in Section~\ref{sec:few}.

\begin{qu}
Is there a function $f$ such that the circuits of any matroid $M$ can be covered by $O(f(\beta(M))\cdot r(M))$ edges?
\end{qu}

%%%%%%%%%%%%%%%%%%%%%%%%
\paragraph{Acknowledgement.} Tamás Schwarcz was supported by the \'{U}NKP-22-3 New National Excellence Program of the Ministry for Culture and Innovation from the source of the National Research, Development and Innovation Fund. The work was supported by the Lend\"ulet Programme of the Hungarian Academy of Sciences -- grant number LP2021-1/2021 and by the Hungarian National Research, Development and Innovation Office -- NKFIH, grant number FK128673. 
%%%%%%%%%%%%%%%%%%%%%%%%

%%%%%%%%%%%%%%%%
\bibliographystyle{abbrv}
\bibliography{sbro}
%%%%%%%%%%%%%%%%

\end{document}